\documentclass[12pt,leqno]{amsart}
\usepackage{amsmath,amssymb,amsfonts, amscd,   
pifont, amsthm,  amsxtra, latexsym, 
eufrak}
\usepackage[all]{xy}
\usepackage{here}
\usepackage[dvips]{graphicx}

\usepackage{psfrag}
\usepackage{amsmath}
\usepackage{amsthm}
\usepackage{amsfonts}
\usepackage{amssymb}
\theoremstyle{plain}
\newtheorem{thm}{Theorem}[section]
\newtheorem{theorem}[thm]{Theorem}

\newtheorem{proposition}[thm]{Proposition}

\newtheorem{lemma}[thm]{Lemma}

\newtheorem{fact}[thm]{Fact}

\theoremstyle{definition}
\newtheorem{definition}[thm]{Definition}

\newtheorem{example}[thm]{Example}

\makeatletter
\@addtoreset{equation}{section}
\def\@seccntformat#1{\csname the#1\endcsname.\quad}
\def\@seccntformat#1{\csname the#1\endcsname%
\expandafter\ifx\csname#1\endcsname\subsection.\fi\quad}
\makeatother
\newlength{\barlength}
\newcommand{\minusfill}{$\mathsurround=0pt\mathord- \mkern-6mu
   \cleaders\hbox{$\mkern-3mu \mathord- \mkern-3mu$}\hfill
     \mkern-6mu \mathord-$}
\newcommand{\yokobo}{\hbox to 1.2em{\minusfill}}
\newcommand{\equalfill}{$\mathsurround=0pt\mathord= \mkern-6mu
    \cleaders\hbox{$\mkern-3mu \mathord= \mkern-3mu$}\hfill
       \mkern-6mu \mathord=$}
\newcommand{\Longlongrightarrow}
       {\hbox to 2em{\equalfill$\mkern-3mu\Rightarrow$}}
\newcommand{\Longlongleftarrow}
       {\hbox to 2em{$\Leftarrow\mkern-3mu$\equalfill}}
\newcommand{\tsume}{\kern-.35em}
\newlength{\circlength}
\divide\circlength by 2
\advance\circlength by 0.1em

\begin{document}
\title{Generating operators
 of symmetry breaking 
 --- from discrete to continuous}
\author{Toshiyuki Kobayashi}

\maketitle

\begin{abstract}
Based on the \lq\lq{generating operator}\rq\rq\ of the Rankin--Cohen brackets
 introduced in Kobayashi--Pevzner 
 [arXiv:2306.16800], 
 we present a method to construct
 various fundamental operators with continuous parameters
 such as invariant trilinear forms
 on infinite-dimensional representations, 
 the Fourier and the Poisson transforms
 on the anti-de Sitter space, 
 and integral symmetry breaking operators
 for the fusion rules,  among others, 
out of  a countable set of differential symmetry breaking operators.  
\end{abstract}



\section{Introduction}
\label{sec:Intro}

This article is a continuation of \cite{KPgen}, 
 where we initiated 
a new line of investigation 
 on branching problems for the restriction of representations
 through \lq\lq{generating operators}\rq\rq\
 between two manifolds.  
By definition, 
 \lq\lq{generating operators}\rq\rq\
 are built on discrete data, 
 which are a countable family of SBOs 
({\it{symmetry breaking operators}})
 in the case we consider.  
On the other hand, 
 general irreducible decompositions
 such as branching laws or Plancherel-type theorems
 often involve continuous spectrum.  
The novel feature of this article is an introduction
 of a method to {\it{transfer}} differential SBOs
 ({\bf{countable data}}) 
 into a meromorphic family of non-local intertwining operators
 ({\bf{continuous data}}).  
We illustrate the trick by the \lq\lq{generating operators}\rq\rq\
 of the Rankin--Cohen brackets.  
In particular, by using the \lq\lq{boundary values}\rq\rq\ 
 of \lq\lq{generating operators}\rq\rq,
 we are able to treat non-holomorphic representations
 in different geometric settings from
 the initial one.

Another novel feature of this article
 is an application of differential SBOs 
 to construct a geometric embedding
 of discrete series representations
 for the de Sitter space $\operatorname{d S}^2$
 into principal series representations (Theorem \ref{thm:23062904}).  
This is carried out 
 by proving the analytic continuation 
of $L^2$-eigenfunctions
 to the boundary of the conformal compactification.  
We remark 
 that such an analytic continuation does NOT exist 
 for eigenfunctions
 that are not square integrable.  
The proof uses the theory of discretely decomposable restriction
\cite{xkInvent98}.

We illustrate the idea by the generating operator
 of the Rankin--Cohen brackets.  
The results suggest 
 that the application of the \lq\lq{generating operators}\rq\rq\
 of SBOs is already rich 
 in the $SL_2$ case.  
See also \cite{KP24}.  
In a subsequent paper,
 we plan to discuss a generalization
 of some of the aspects discovered here 
 to other reductive groups
 of higher rank.

The article is organized as follows.  
In Section \ref{sec:GSBO}, 
 we explain the geometric setting
 of branching problems 
 for which the \lq\lq{generating operators}\rq\rq\
 make sense by taking the general fusion rules
 as an example. 
Section \ref{sec:SL} discusses a trick from 
 \lq\lq{discrete}\rq\rq\ to \lq\lq{continuous}\rq\rq\
 via the generating operators, 
 for which the representation theoretic setting also changes from 
 \lq\lq{discretely decomposable restrictions}\rq\rq\
 to \lq\lq{branching laws with continuous spectrum}\rq\rq.  
Sections \ref{sec:dS} and \ref{sec:5}
 extend the results 
beyond symmetry breaking 
 through the conformal compactification.  

\section{Generalities: generating operators for SBOs}
\label{sec:GSBO}

This section reveals a representation-theoretic background
 of the closed formula
 of the generating operators
 \cite{KPgen}
 in a specific setting, 
 and investigates a more general setting
 in branching problems
 for which we could expect a further detailed study
 of the \lq\lq{generating operators}\rq\rq.  
The key requirements are discrete decomposability
 and multiplicities 
 of the restriction.  
For simplicity, 
 we confine ourselves 
 to the tensor product case.

\subsection{Generating operators for symmetry breaking operators}
Let $X$ and $Y$ be two manifolds.  
For a family of linear maps
 $R_{\ell} \colon \Gamma(X) \to \Gamma(Y)$
 between the spaces of functions on $X$ and $Y$, 
 the \lq\lq{generating operator}\rq\rq\
 $T$ is defined as a $\operatorname{Hom}_{\mathbb{C}}(\Gamma(X), \Gamma(Y))$-valued
 formal power series of $t$, 
 see \cite{KPgen}:
\begin{equation}
\label{eqn:genop}
  T=\sum_{\ell=0}^{\infty} \frac{R_{\ell}}{\ell!} t^{\ell}
  \in \operatorname{Hom}_{\mathbb{C}}(\Gamma(X), \Gamma(Y)) \otimes {\mathbb{C}}[[t]].  
\end{equation}

Very special cases of the generating operators  include 
 the classical notion
 of the generating functions
 of orthogonal polynomials
 and the semigroups generated 
 by self-adjoint operators
 {\it{e.g.,}} the Hille--Yosida theory.

Suppose that a pair of groups $\widetilde G \supset G$ act on $X \supset Y$, 
respectively, 
 and that $R_{\ell} \colon \Gamma(X) \to \Gamma(Y)$
 $(\ell \in {\mathbb{N}})$ are a family of SBOs 
 ({\it{symmetry breaking operators}}), 
 {\it{i.e.}}, 
 each $R_{\ell}$ is a $G$-intertwining operator
into a multiplier representation $(\pi_{\ell}, \Gamma(Y))$ of $G$.  
We are particularly interested in the setting
 where the $\widetilde G$-module $\Gamma(X)$ decomposes {\bf{discretely}} 
 into irreducible representations $\pi_{\ell}$ of $G$
 with {\bf{bounded multiplicity}}.

In the rest of this section, 
 we summarize briefly some recent developments 
 about when such settings arise, 
 with focus on the case 
 of tensor product representations 
 of $G$, 
 in other words, 
 the case where $\widetilde G=G \times G$.

\subsection{Generalities: Multiplicity of the tensor product}
~~~
\newline
In defining the generating operator of SBOs, 
 it would be natural to impose some control of {\it{multiplicities}}
 in the branching laws.

Let $G$ be a real reductive Lie group, 
 ${\mathcal{M}}(G)$ the category
 of finitely generated, 
 smooth admissible representations of $G$
 of moderate growth, 
 see \cite[Chap.\ 11]{xWaII}, 
 and $\operatorname{Irr}(G)$ the set of irreducible objects in ${\mathcal{M}}(G)$.  

\begin{definition}
[multiplicity]
For $\pi_1, \pi_2, \tau \in \operatorname{Irr}(G)$, 
 the multiplicity of $\tau$
 in the tensor product representation $\pi_1 \otimes \pi_2$ is defined by 
\[
[\pi_1 \otimes \pi_2:\tau]
:=
\dim_{\mathbb{C}} 
{\operatorname{Hom}_{G}
(\pi_1 \otimes \pi_2, \tau)}
\in {\mathbb{N}} \cup \{\infty\}, 
\]
where $\operatorname{Hom}_G(\,,\,)$ denotes the space
 of SBOs 
 ({\it{i.e.,}} continuous
\newline
 $G$-homomorphisms)
 between the Fr{\'e}chet representations.  
\end{definition}

The finiteness condition 
 of the multiplicity $[\pi_1 \otimes \pi_2: \tau]$
 gives a strong constraint on the group $G$:

\begin{fact}
[{\cite{Ksuron}, see also \cite{xKMt}}]
\label{fact:2.2}
The following three conditions on a non-compact simple Lie group $G$
 are equivalent.  
\\
{\rm{(i)}}\enspace
$[\pi_1 \otimes \pi_2:\tau]<\infty$
\quad
for any $\pi_1$, $\pi_2$, $\tau \in \operatorname{Irr}(G)$.  
\\
{\rm{(ii)}}\enspace
The triple product of real flag varieties $G/P$ is real spherical.  
\\
{\rm{(iii)}}\enspace
${\mathfrak{g}} \simeq {\mathfrak{s o}}(n,1)$.  
\end{fact}

The proof includes 
 that the tensor product $\pi_1 \otimes \pi_2$
 is of infinite multiplicity
 for \lq\lq{generic representations}\rq\rq\ $\pi_1$ and $\pi_2$
 except 
 when ${\mathfrak{g}} \simeq {\mathfrak{so}}(n,1)$.  
In contrast, 
 if $\pi_1$ and $\pi_2$ are \lq\lq{sufficiently small}\rq\rq\ infinite-dimensional representations of $G$, 
 the multiplicity in $\pi_1 \otimes \pi_2$ may stay finite, 
 see \cite{mf-korea, TK22, TK23}
 for precise formulation.  
In particular, 
 one has:
\begin{fact}
[{\cite{TK23}}]
\label{fact:bddtensor}
For any 1-connected non-compact simple Lie group $G$, 
 there always exist infinite-dimensional
 irreducible representations $\pi_1$, $\pi_2$ of $G$
 such that $\pi_1 \otimes \pi_2$ is of uniformly bounded multiplicity:\begin{equation}
\label{eqn:bddtensor}
\underset{\tau \in \operatorname{Irr}(G)}{\operatorname{sup}}
[\pi_1 \otimes \pi_2:\tau] < \infty.  
\end{equation}
\end{fact}

\subsection{Generalities: Discrete decomposability of restriction}
~~~
\newline
Another important requirement in defining 
 the \lq\lq{generating operator}\rq\rq\
 of SBOs
 is the discrete decomposability
 of the restriction.  
Applying the general criterion 
 \cite{xkInvent94, xkAnn98, xkInvent98}
 to the tensor product case, 
 one has  

\begin{fact}
\label{fact:disctensor}
Let $\pi_1$, $\pi_2$ be two infinite-dimensional irreducible representations
 of a simple Lie group $G$.  
\begin{enumerate}
\item[{\rm{(1)}}]
{\rm{({\cite[Thm.\ 6.1]{xkyosh13}})}}\enspace
The following two conditions on the triple
\\
 $(G, \pi_1, \pi_2)$
 are equivalent:
\\
{\rm{(i)}}\enspace
$\pi_1 \otimes \pi_2$ is discretely decomposable.  
\\
{\rm{(ii)}}\enspace
$G/K$ is a Hermitian symmetric space
 and $\pi_1$, $\pi_2$ are
simultaneously
 highest (or lowest) weight modules.  

\item[{\rm{(2)}}]
{\rm{({\cite{mf-korea}})}}\enspace
If one of (therefore both of) these conditions
 is satisfied, 
 then the uniformly bounded multiplicity property \eqref{eqn:bddtensor} holds.  \end{enumerate}
\end{fact}

The representations $\pi_1$ and $\pi_2$ in Fact \ref{fact:disctensor}
 can be realized
 in the holomorphic category, 
 see {\it{e.g.,}} \cite{mf-korea},
 for which structural results 
 of SBOs are investigated in \cite{KP1}
 such as the {\it{localness theorem}} 
 and the {\it{extension theorem}}.

In \cite{KPgen}, 
 the generating operators of SBOs are explored
 in a special case
 of the general framework of discretely decomposable restrictions 
 with bounded multiplicities, 
 as discussed in Fact \ref{fact:2.2}
 and Fact \ref{fact:disctensor}.

\section{From Discrete data to continuous data}
\label{sec:SL}
This section illustrates 
 by an $SL_2$ example
 how the generating operators transfer 
 {\bf{discrete data}}
 into {\bf{continuous data}}.  
The diagram
\[
   \{R_{\ell}\}_{\ell \in {\mathbb{N}}}
   \,\,\dashrightarrow\,\,
   T
   \,\,\dashrightarrow\,\,
   T_{\mu}^{\pm}, {\mathcal{P}}_{\lambda}^{\pm}, {\mathcal{F}}_{\lambda}^{\pm}
\]
indicates that the closed formula  \eqref{eqn:Luminy230314}
 of the {\it{generating operator}} $T$
 of the Rankin--Cohen brackets
 $\{R_{\ell}\}_{\ell \in {\mathbb{N}}}$ is a key
 to reproduce explicit formul{\ae} of various families
 of non-local intertwining operators
 with continuous parameter
 such as

$\bullet$\enspace
symmetry breaking operators $T_{\mu}^{\pm}$ for the fusion rule
 of the Hardy spaces
 (or invariant trilinear forms)
 (Proposition \ref{prop:23070233});

$\bullet$\enspace
Poisson transforms ${\mathcal{P}}_{\lambda}^{\pm}$
 for the de Sitter space 
 (Proposition \ref{prop:23070306});

$\bullet$\enspace
Fourier transforms ${\mathcal{F}}_{\lambda}^{\pm}$
 on the de Sitter space 
 (Proposition \ref{prop:23070311}).  

\vskip 1pc
We note that these intertwining operators are already known, 
 {\it{e.g.,}} in \cite{xksbonvec, toshima}
 in a more general setting
 by other approaches.  
The novelty here is 
 that the distribution kernels of these non-local operators are 
 explicitly reconstructed from a countable family of differential operators
 on a different geometry 
 through the {\it{generating operator}} $T$.  
It should be noted 
 that this is {\it{opposite}}
 to the usual direction 
such as taking the residues 
 of the meromorphic family of non-local operators.

In this article, 
 we focus on this new trick, 
 and omit the proof of some standard statements
 such as the meromorphic continuation 
or the covariance property, 
 which can be proven by existing techniques, 
 {\it{e.g.,}} \cite{CKOP, xksbonvec, toshima}.  
Our approach is formulated by viewing elements
 in principal series representations
 as local cohomologies, 
 or \lq\lq{boundary values}\rq\rq\
 of holomorphic functions
 in the spirit of Sato's hyperfunctions
 \cite{xsato}.

\subsection{Preliminaries: Principal series representations of $G$}
~~~
\newline
We fix some notation for representations
 of $G=S L(2,{\mathbb{R}})$.  
Take a minimal parabolic subgroup $P$ 
 to be the set of lower triangular matrices, 
 and define characters
 $\chi_{\lambda}^+$
 and $\chi_{\lambda}^-$ of $P$, respectively, 
 by 
\[
   \text{$\chi_{\lambda}^+\begin{pmatrix} a & 0 \\ c & a^{-1} \end{pmatrix} 
   :=|a|^{-\lambda}$, 
\quad
  $\chi_{\lambda}^-\begin{pmatrix} a & 0 \\ c & a^{-1} \end{pmatrix} 
   := |a|^{-\lambda}\operatorname{sgn}a$.}
\]
Let ${\mathcal{L}}_{\lambda}\equiv{\mathcal{L}}_{\lambda}^+$ and 
 ${\mathcal{L}}_{\lambda}^-$
 be the homogeneous line bundles over $G/P$
 associated to the characters $\chi_{\lambda}^+$ and $\chi_{\lambda}^-$, 
 respectively.  
The natural action of $G$ 
 on $C^{\infty}(G/P, {\mathcal{L}}_{\lambda}^{\pm})$
 defines the principal series representations.  
By using the Bruhat decomposition, 
 they are expressed as the multiplier representations:
 for $\varepsilon \in \{+, -\}\equiv \{1, -1\}$, 
\begin{equation}
\label{eqn:ps}
  (\varpi_{\lambda}^{\varepsilon}(g)f)(x)=|cx+d|^{-\lambda}\operatorname{sgn}(cx+d)^{\frac{1-\varepsilon}{2}}f(\frac{ax+b}{cx+d})
\end{equation}
for $g^{-1}=\begin{pmatrix} a & b \\ c & d \end{pmatrix}$.

\subsection{Generating operator for Rankin--Cohen brackets}
~~~
\newline
Let $Q(\zeta_1, \zeta_2;z, t)$ be a holomorphic function 
 of four variables
 given by 
\begin{equation}
\label{eqn:Q}
   Q(\zeta_1, \zeta_2;z, t):=(\zeta_1-z)(\zeta_2-z)+t(\zeta_1-\zeta_2).  
\end{equation}
In \cite{KPgen}, 
 we introduced an integral transform
$
   T \colon {\mathcal{O}}({\mathbb{C}}^2) \to {\mathcal{O}}({\mathbb{C}}^2)
$
 by 
\begin{equation}
\label{eqn:Luminy230314}
(T f)(z,t):=
\frac{1}{(2 \pi \sqrt{-1})^2}\oint_{C_1}\oint_{C_2} 
\frac{f(\zeta_1, \zeta_2)}{Q(\zeta_1, \zeta_2;z, t)} d \zeta_1 d \zeta_2, 
\end{equation}
where $C_j$ are sufficiently small contours
 around the point $z$ ($j=1,2$).  
It is proven in \cite[Thm.\ 2.3]{KPgen}
 that $T$ is the \lq\lq{generating operator}\rq\rq\
 of the family of the Rankin--Cohen brackets $\{R_{\ell}\}_{\ell \in {\mathbb{N}}}$, 
 see \cite{xRa56}, 
 namely, 
\begin{equation}
\label{eqn:TRC}
   (T f)(z,t)=\sum_{\ell =0}^{\infty} \frac{t^{\ell}}{\ell !}R_{\ell} f(z)
\quad
  \text{for any $f \in {\mathcal{O}}({\mathbb{C}}^2)$, }  
\end{equation}
 where 
$
  R_{\ell} \colon {\mathcal{O}}({\mathbb{C}}^2) \to 
  {\mathcal{O}}({\mathbb{C}})
$,
 $f(\zeta_1, \zeta_2) \mapsto (R_{\ell}f)(z)$
 is defined by 
\begin{equation}
\label{eqn:RC}
  (R_{\ell} f)(z):=\sum_{j=0}^{\ell} (-1)^j \left(\ell \atop j\right)^2 
 \left. \frac{\partial^{\ell} f(\zeta_1, \zeta_2)}{\partial \zeta_1^{\ell-j} \partial \zeta_2^j}\right|_{\zeta_1=\zeta_2=z}
\quad
\text{for $\ell \in {\mathbb{N}}$.  }
\end{equation}

The operators $\{R_{\ell}\}_{\ell \in {\mathbb{N}}}$ are differential SBO
 for the fusion rule 
of the two Hardy spaces, 
see \cite{vDP, KPgen, xRa56} for instance.

In this case, 
 the formal power series \eqref{eqn:TRC} converges uniformly
 on any compact set in ${\mathbb{C}}^2$.  
Conversely, 
 every operator $R_{\ell}$ ($\ell \in {\mathbb{N}}$) is recovered readily from the generating operator $T$ by
\begin{equation}
\label{eqn:tTRl}
R_{\ell}=\left.\left(\frac{\partial}{\partial t}\right)^{\ell}\right|_{t=0}
 \circ T.  
\end{equation}

\subsection{From \lq\lq{discrete}\rq\rq\ to \lq\lq{continuous}\rq\rq}
~~~
\newline
This section defines a \lq\lq{meromorphic extension}\rq\rq\ 
 of the countable family $\{R_{\ell}\}$ 
 of operators 
 in the spirit of fractional calculus.  
We construct operators $T_{\mu}^{\pm}$
 that depend meromorphically on $\mu$
 and the residue operator 
 is equal to $R_{\ell}$ 
 up to scalar multiplication for every $\ell \in {\mathbb{N}}$, 
 see \eqref{eqn:restT}.

We begin by recalling a classical fact that 
\[
  t_+^{\mu}:=
  \begin{cases}
  t^{\mu}\quad&(t>0)
\\
  0&(t \le 0), 
  \end{cases}
\qquad
t_-^{\mu}:=(-t)_+^{\mu}
\]
 are locally integrable functions on ${\mathbb{R}}$ 
 for $\operatorname{Re} \mu>-1$, 
 and extend to tempered distributions
 which depend meromorphically on $\mu \in {\mathbb{C}}$.  
Their poles are all simple and have the following residues:
\begin{equation}
\label{eqn:23072202}
   \left.\left(\frac{\partial}{\partial t}\right)^{\ell}\right|_{t=0}
  =\left.\frac{(-1)^{\ell}}{\Gamma(\mu+1)} t_+^{\mu}\right|_{\mu=-\ell-1}
  =\ell! \underset{\mu=-\ell-1}{\operatorname{res}} t_+^{\mu}.  
\end{equation}

Let $f(\zeta_1, \zeta_2) \in {\mathcal{O}}({\mathbb{C}}^2)$.  
Inspired by the fomul{\ae}
 \eqref{eqn:tTRl} and \eqref{eqn:23072202}, 
 we define a \lq\lq{meromorphic continuation}\rq\rq\
 of $(R_{\ell}f)(z)$ by setting 
\begin{align}
\notag
  (T_{\mu}^{\pm}f)(z)
  :=&
   \langle t_{\pm}^{\mu}, T f(z,t) \rangle
\\
\label{eqn:Tmupm}
=&\frac{1}{(2 \pi \sqrt{-1})^2}
 \int_{{\mathbb{R}}}t_{\pm}^{\mu}
 \left(\oint_{C_1} \oint_{C_2} \frac{ f (\zeta_1, \zeta_2)}{Q(\zeta_1, \zeta_2;z, t)} d \zeta_1 d \zeta_2\right)
  d t.
\end{align}

Our integral formula \eqref{eqn:Luminy230314} of the generating operator $T$
 is formulated originally in the holomorphic category.  
We now interpret principal series representations
via local cohomologies of holomorphic functions
 ({\it{e.g.,}} \lq\lq{boundary values}\rq\rq\ in the one variable case).  
We proceed by changing the order of the integration
 in \eqref{eqn:Tmupm}.  
We set 
\begin{equation}
\label{eqn:23070230}
 K_{\pm}^{\mu}(\zeta_1, \zeta_2; \zeta)
:=\left(\frac{(\zeta_1-\zeta)(\zeta_2-\zeta)}{\zeta_1-\zeta_2}\right)_{\pm}^{\mu}
\end{equation}
as hyperfunctions
 depending meromorphically on $\mu \in {\mathbb{C}}$.  
By Lemma \ref{lem:23062411} in Appendix, 
 we have:
\begin{lemma}
\label{lem:23062413}
One has
\begin{equation*}
\langle t_{\pm}^{\mu}, \frac{1}{Q(\zeta_1, \zeta_2; \zeta, t)}\rangle
=
\frac{-2 \pi \sqrt{-1}}{\zeta_1-\zeta_2}
K_{\mp}^{\mu}(\zeta_1, \zeta_2; \zeta).  
\end{equation*}

\end{lemma}

In what follows, 
we use the notation ${\mathcal{L}}={\mathcal{L}}_{1}^{-}$
 and ${\mathcal{L}}_{-2\mu}={\mathcal{L}}_{-2\mu}^{+}$
 for simplicity.   
By Lemma \ref{lem:23062413}, 
 we have:
\begin{proposition}
\label{prop:23070233}
For $f \in C^{\infty}(G/P \times G/P, {\mathcal{L}} \boxtimes {\mathcal{L}})$, 
$(T_{\mu}^{\pm}f)$ takes the following form
\begin{equation}
\label{eqn:23070233}
(T_{\mu}^{\pm} f) (\zeta)=\frac{-1}{2 \pi \sqrt{-1}}\int_{{\mathbb{R}}^2} f (\zeta_1, \zeta_2) K_{\mp}^{\mu} (\zeta_1, \zeta_2; \zeta) \frac{d \zeta_1 d \zeta_2}{\zeta_1-\zeta_2},  
\end{equation}
and defines a family of symmetry breaking operators 
\[
   T_{\mu}^{\pm} \colon 
   C^{\infty}(G/P \times G/P, {\mathcal{L}} \boxtimes {\mathcal{L}})
   \to
   C^{\infty}(G/P, {\mathcal{L}}_{-2\mu})
\]
 which depend meromorphically on $\mu \in {\mathbb{C}}$.  
Moreover, 
 one has
\begin{equation}
\label{eqn:restT}
  \underset{\mu=-\ell-1}{\operatorname{res}} T_{\mu}^{\pm}f = \frac{1}{\ell !} R_{\ell} f. 
\end{equation}
\end{proposition}

We note
 that $R_{\ell}$ extends to $G/P \times G/P$
 by the extension theorem 
 on differential SBOs
 in the general setting, 
 see \cite[Thm.\ B]{KP1}.  
The residue formula \eqref{eqn:restT} follows directly from \eqref{eqn:23072202}, 
 or alternatively from the lemma below.  
\begin{lemma}
\label{lem:23061910}
For any $\ell \in {\mathbb{N}}$
 and for any $f \in {\mathcal{O}}({\mathbb{C}}^2)$, 
\begin{equation*}
\left.
\frac{\partial^{2\ell}}{\partial \zeta_1^{\ell}\partial \zeta_2^{\ell}}
\right|_{\zeta_1 = \zeta_2 =\zeta} 
((\zeta_1-\zeta_2)^{\ell}f)
=(-1)^{\ell} \ell!(R_{\ell}f) (\zeta).  
\end{equation*}
\end{lemma}

\begin{proposition}
[Holographic operator]
\label{prop:23070307}
As the dual operator of $T_{\mu}^{\pm}$ 
 (up to scalar multiplication by $-2\pi \sqrt{-1}$), 
\[
  H_{\mu}^{\pm}
  \colon 
  {\mathcal{D}}'(G/P, {\mathcal{L}}_{2\mu+2})
  \to 
  {\mathcal{D}}'(G/P \times G/P, {\mathcal{L}} \boxtimes {\mathcal{L}})
\]
gives a family
 of $G$-intertwining operators
 depending meromorphically on $\mu \in {\mathbb{C}}$.  
The operators take the following form:
\begin{equation}
\label{eqn:23070307}
 (H_{\mu}^{\pm} h)(\zeta_1, \zeta_2)
  =\frac{1}{\zeta_1-\zeta_2}
  \int_{\mathbb{R}}h(\zeta) K_{\mp}^{\mu}(\zeta_1, \zeta_2;\zeta)d \zeta.  
\end{equation}
\end{proposition}

\section{From Rankin--Cohen brackets to Poisson transforms}
\label{sec:dS}
This section gives yet another example from \lq\lq{{\bf{discrete}}}\rq\rq\
 to \lq\lq{{\bf{continuous}}}\rq\rq.  
We shall see 
 that the Rankin--Cohen brackets $\{R_{\ell}\}_{\ell \in {\mathbb{N}}}$ yields
 a pair of the Poisson transforms ${\mathcal{P}}_{\lambda}^{\pm}$ 
 on the de Sitter space $\operatorname{d S}^2$
 via the \lq\lq{generating operator}\rq\rq\ $T$
 in \eqref{eqn:Luminy230314}.  
Our strategy is to restrict the holographic operators
$
    H_{\mu}^{\pm}
$
 in Proposition \ref{prop:23070307}, 
 summarized as 
\[
  \{R_{\ell}\}_{\ell \in {\mathbb{N}}}
  \,\,\rightsquigarrow\,\,
  T
  \,\,\rightsquigarrow\,\,
  T_{\mu}^{\pm} 
  \,\,\rightsquigarrow\,\,
  H_{\mu}^{\pm} 
  \,\,\rightsquigarrow\,\,
  {\mathcal{P}}_{\lambda}^{\pm}.  
\]

\subsection{Bruhat coordinates of $\operatorname{dS}^2$}
~~~
\newline
The de Sitter space $\operatorname{dS}^2$ is a Lorentzian manifold
 with constant curvature $+1$, 
 defined as a surface of the Minkowski space ${\mathbb{R}}^{2,1}$:
\[
  \operatorname{d S}^2=\{(x,y,z) \in {\mathbb{R}}^3: x^2 +y^2-z^2=1\}.  
\]
We may realize $\operatorname{d S}^2$ in the matrix form
\[
\{A=\begin{pmatrix} x & y+z \\ y-z & -x \end{pmatrix}:
 \det A=-1\}\subset {\mathfrak{sl}}(2,{\mathbb{R}}), 
\]
on which $G=SL(2,{\mathbb{R}})$ acts via the adjoint representation.  
Let 
\[
  I_{1,1}:=\begin{pmatrix} 1 & \\ & -1\end{pmatrix} \in {\mathfrak{s l}}(2,{\mathbb{R}}), \quad
  H:=\{\begin{pmatrix} a & \\ & a^{-1}\end{pmatrix}: a \in {\mathbb{R}}^{\times}\} \subset G.  
\]
Then $\operatorname{d S}^2$ is identified 
 with the homogeneous space $G/H$
 by 
\[
  G/H \overset \sim \rightarrow \operatorname{d S}^2, 
\quad
g H \mapsto \operatorname{Ad}(g)I_{1,1}
 =\begin{pmatrix} a d + b c & -2 a b \\ 2 c d & -(a d + b c)\end{pmatrix}
\]
 where $g=\begin{pmatrix} a & b \\ c & d \end{pmatrix}$.  
In the coordinates, 
 one has 
\begin{equation}
\label{eqn:abcdxyz}
  (x,y,z)=(ad+bc, -ab+cd, -ab-cd).  
\end{equation}

The third realization of $\operatorname{d S}^2$ is given 
 via the $G$-orbit decomposition 
\begin{equation}
\label{eqn:GPGP}
  G/P \times G/P = \operatorname{d S}^2 \,\amalg\,\, G/P
\qquad
\text{(disjoint)}
\end{equation}
under the diagonal action of $G$.  
Let $w:=\begin{pmatrix} 0 & -1 \\ 1 & 0 \end{pmatrix}$.  
Since $P \cap w Pw^{-1}=H$, 
 the $G$-orbit through $(e P, w P)$ is identified with $G/H$.  
Combining this with the Bruhat decomposition $G/P={\mathbb{R}} \cup \{\infty\}$, 
 one has the diagram below: 
\begin{alignat}{6}
\label{eqn:XRtwo}
  \operatorname{d S}^2 \hphantom{M} &\,\,\simeq\,\, &&G/H &&\hookrightarrow &&G/P \times G/P &&\hookleftarrow &&{\mathbb{R}}^2
\\
\notag
\operatorname{Ad}(g)I_{1,1} 
&\,\,\rotatebox[origin=c]{180}{$\mapsto$}\,\, 
&& g H && \mapsto &&(gP, g w P) &&\rotatebox[origin=c]{180}{$\mapsto$}
 &&(\zeta_1, \zeta_2).  
\end{alignat}
Then $(x,y,z) \in \operatorname{d S}^2$ has the following coordinates
 by \eqref{eqn:abcdxyz} and \eqref{eqn:XRtwo}:
\begin{equation}
\label{eqn:zetaxyz}
  (\zeta_1, \zeta_2)=(-\frac{y+z}{x+1}, \frac{x+1}{y-z}).  
\end{equation}

It is convenient to list some elementary formul{\ae}  
 concerning \eqref{eqn:zetaxyz}: 
\begin{lemma}
\label{lem:23070130}
Retain the setting as above.  
One has 
\begin{align}
\label{eqn:23062715}
\zeta_1-\zeta_2=&\frac{-2}{y-z}, 
\\
\label{eqn:23062714}
(\zeta_1+\sqrt{-1})(\zeta_2+\sqrt{-1})=&\frac{2 \sqrt{-1} (x+\sqrt{-1} y)}{y-z}, 
\end{align}
\begin{equation*}
\frac{\zeta_1-\zeta_2}{(\zeta_1-\zeta)(\zeta_2-\zeta)}
=\frac{2 (1+x)}{((1+x)\zeta+(y+z))((1+x)-(y-z)\zeta)}.  
\end{equation*}
\end{lemma}

The Minkowski metric $d s^2 = d x^2 + d y^2 - d z^2$
 on ${\mathbb{R}}^{2,1}$ induces an invariant measure on the de Sitter space $\operatorname{d S}^2$  as below.

\begin{lemma}
\label{lem:23070226}
In the coordinates
 $(x,y,z)=(\cosh t \cos \theta, \cosh t\sin \theta, \sinh t)$
 and \eqref{eqn:zetaxyz}, 
 the invariant measure on $\operatorname{d S}^2$ 
 takes the following form:
\begin{equation*}
\frac{d x d y}{2 z} 
= \cosh t d t d \theta
= \frac {2}{(\zeta_1-\zeta_2)^2}d \zeta_1 d \zeta_2.  
\end{equation*}
\end{lemma}

\subsection{Tensor product of principal series and $C^{\infty}(G/H)$}
~~~
\newline
The open embedding \eqref{eqn:GPGP}
 of the de Sitter space $\operatorname{d S}^2$
 in $G/P \times G/P$ connects the tensor product 
 of two principal series representations
 of the group $G$
 with the harmonic analysis on $\operatorname{d S}^2 \simeq G/H$:
\begin{lemma}
\label{lem:23062717}
For any $\lambda \in {\mathbb{C}}$ and $\varepsilon \in \{+, -\}$, 
 the line bundle
 ${\mathcal{L}}_{\lambda}^{\varepsilon} \boxtimes {\mathcal{L}}_{\lambda}^{\varepsilon}$
 becomes trivial as a $G$-equivariant bundle
 when restricted to the submanifold $G/H$.  
Accordingly, 
 the pull-back induces a $G$-homomorphism
\[
  \iota_{\lambda}^{\ast}\colon C^{\infty}(G/P \times G/P, 
  {\mathcal{L}}_{\lambda}^{\varepsilon} \boxtimes {\mathcal{L}}_{\lambda}^{\varepsilon})
\hookrightarrow 
 C^{\infty}(G/H), 
\]
\[
f(\zeta_1, \zeta_2) \mapsto F(x, y, z)
=  (\frac{2}{z-y})^{\lambda} 
   f(-\frac{y+z}{x+1}, \frac{x+1}{y-z})
=(\zeta_1-\zeta_2)^{\lambda}
 f(\zeta_1, \zeta_2).  
\]

\end{lemma}

\subsection{Poisson transforms on $\operatorname{d S}^2$}
~~~
\newline
We define the Poisson transforms
 as the composition 
$
  {\mathcal{P}}_{\lambda}^{\pm}:=
  \iota_1^{\ast} \circ H_{\frac \lambda 2-1}^{\pm}.  
$
By Lemma \ref{lem:23062717} and by \eqref{eqn:23070230}, 
 ${\mathcal{P}}_{\lambda}^{\pm}$ takes the form 
\[
  ({\mathcal{P}}_{\lambda}^{\pm}h)(x,y,z)
  =
  \int_{\mathbb{R}}
  {\mathcal{K}}_{\mp}^{\frac \lambda 2 -1}(x,y,z;\zeta) h (\zeta) d \zeta, 
\]
where ${\mathcal{K}}_{\pm}^{\mu}$ is the pull-back of $K_{\pm}^{\mu}$ 
 in \eqref{eqn:23070230} from $(G/P)^3$ to $G/H \times G/P$, 
 see \eqref{eqn:XRtwo} and \eqref{eqn:zetaxyz}.  
By Lemma \ref{lem:23070130}, 
 ${\mathcal{K}}_{\pm}^{\mu}$ amounts to 
\begin{equation*}
{\mathcal{K}}_{\pm}^{\mu}(x,y,z;\zeta)
=\left(\frac{((1+x)\zeta+(y+z))((1+x)-(y-z)\zeta)}{2(1+x)}\right)
_{\pm}^{\mu}.  
\end{equation*}

Let $\Delta$ be the Laplacian on $\operatorname{d S}^2$
 with respect to the Lorentzian metric 
 induced from the Minkowski space ${\mathbb{R}}^{2,1}$.  
For $\Gamma= C^{\infty}$, $L^2$, $\ldots$, 
 we set
\begin{equation}
\label{eqn:23072114}
   {\mathcal{F}}(G/H, {\mathcal{M}}_{\lambda})
 :=\{f \in \Gamma(G/H)
     : \Delta f = -\frac 1 4 \lambda (\lambda -2) f\}.  
\end{equation}

\begin{proposition}
[Poisson transform]
\label{prop:23070306}
The transform 
\[
  {\mathcal{P}}_{\lambda}^{\pm}
\colon 
  C^{\infty}(G/P, {\mathcal{L}}_{\lambda})
  \to 
  C^{\infty}(G/H, {\mathcal{M}}_{\lambda}) \subset C^{\infty}(G/H)  
\]
define $G$-intertwining operators
 that depend meromorphically on $\lambda \in {\mathbb{C}}$.  
\end{proposition}

We note that there are two Poisson transforms ${\mathcal{P}}_{\lambda}^+$ and ${\mathcal{P}}_{\lambda}^-$
 in our setting
 because the $H$-action on $G/P$ has two open orbits, 
 see \cite{toshima}.

\subsection{Fourier transform}
~~~
\newline
The Plancherel formula for $\operatorname{d S}^2 \simeq G/H$
 is known, 
 see \cite{xfar} for instance, 
 which contains both discrete and continuous spectrum:
\begin{equation}
\label{eqn:Pl}
  L^2(G/H) \simeq 
  \underset{\ell =0}{\overset{\infty}\sum}
\raisebox{2ex}{\text{\scriptsize{$\oplus$}}}
  (\pi_{2\ell+2}^+ \oplus \pi_{2\ell+2}^-)
  \oplus 2 \int_{(0,\infty)}^{\oplus} \varpi_{1+\sqrt{-1}\nu} d \nu.  
\end{equation}
In the right-hand side
 $\pi_{2\ell+2}^+$ is the holomorphic discrete series representation
 with minimal $K$-type $\chi_{2\ell+2}$, 
 and $\pi_{2\ell+2}^-$ is its contragredient representation.  
By an abuse of notation, 
 we write $\varpi_{1+\sqrt{-1} \nu}$
 for the spherical unitary principal series representation of $G$, 
 obtained as the unitarization
 of $\varpi_{1+\sqrt{-1} \nu}^+$ of $G$, 
 see \eqref{eqn:ps}.

In this section, 
 we discuss how the generating operator
 of the Rankin--Cohen brackets
 ({\bf{discrete data}})
 is connected with the {\bf{continuous spectrum}}
 in the Plancherel formula
 \eqref{eqn:Pl}
 of $\operatorname{d S}^2$.

As the dual of
 ${\mathcal{P}}_{2-\lambda}^{\pm}$, 
 we define the Fourier transform by
\[
  {\mathcal{F}}_{\lambda}^{\pm}\colon
  C_c^{\infty}(G/H)\to 
  C^{\infty}(G/P, {\mathcal{L}}_{\lambda}).  
\]

\begin{proposition}
[Fourier transform]
\label{prop:23070311}
${\mathcal{F}}_{\lambda}^{\pm}$ takes the form 
\begin{equation}
\label{eqn:23070311}
({\mathcal{F}}_{\lambda}^{\pm}h)(\zeta)
  =
  \int_{G/H}
  {\mathcal{K}}_{\mp}^{-\frac \lambda 2}(x,y,z;\zeta) h (x,y,z) d \mu_{G/H},  
\end{equation}
and one has
$
  {\mathcal{F}}_{\lambda}^{\pm} \circ \iota_1^{\ast}
  =
   T_{-\frac 1 2 \lambda}^{\pm}  
$
 (up to non-zero scalar multiple).  
\end{proposition}

In summary, 
 a countable set 
 of differential SBOs
 (the Rankin--Cohen brackets $\{R_{\ell}\}_{\ell \in {\mathbb{N}}}$)
 led us to the non-local operators
 ${\mathcal{F}}_{\lambda}^{\pm}$ (Fourier transforms) 
 in the framework 
 \lq\lq{from discrete to continuous}\rq\rq\
 via the \lq\lq{generating operator}\rq\rq\ $T$.  
The parameters $\lambda \in 1 + \sqrt{-1}{\mathbb{R}}$
 contribute to the continuous part
 of the Plancherel theorem \eqref{eqn:Pl}.

In Section \ref{subsec:RCdisc}, 
 we shall see that the Rankin--Cohen brackets again show up
 in dealing with the discrete part of \eqref{eqn:Pl}.

\section{Embedding of discrete series into principal series}
\label{sec:5}

Casselman's embedding theorem, 
 see {\it{e.g.,}} \cite{xWaII} tells us
 that every irreducible admissible representation
 of a real reductive group
 can be realized as a subrepresentation
 of some principal series representation.  
However, 
 this abstract theorem does not provide an explicit intertwining operator from 
 a geometric model 
of the irreducible representation
 into a principal series representation.

In this section, 
 we prove that the Rankin--Cohen brackets give geometric embeddings
 of discrete series representations
 of the de Sitter space $\operatorname{d S}^2$
 into principal series representations.  
Since the Rankin--Cohen brackets $R_{\ell}$
 involve the restriction to the diagonal submanifold $G/P$, 
 $R_{\ell}$ is not well defined initially 
 for functions on $\operatorname{d S}^2$
 because $G/P \,\cap\, \operatorname{d S}^2 = \emptyset$, 
 see \eqref{eqn:GPGP}.  
The key ingredients of the proof are
\begin{enumerate}
\item[$\bullet$]
$L^2(\operatorname{d S}^2) \simeq \widehat \pi \otimes \widehat \pi$, 
 see \eqref{eqn:unitarytensor} below, 

\item[$\bullet$]
the theory of admissible restrictions \cite{xkInvent98}, and
\item[$\bullet$]
the extension theorem of differential SBOs \cite{KP1}.  
\end{enumerate}

\subsection{Analytic extension from $\operatorname{d S}^2$ to $G/P \times G/P$}
~~~
\newline
We recall from \cite{xfar} 
 ({\it{cf}}.\ \eqref{eqn:Pl})
 that the space of $L^2$-eigenfunctions
 of the Laplacian splits into the sum
 of two irreducible representations
 of $G$:
\[
   L^2(G/H,{\mathcal{M}}_{2\ell+2}) 
   \simeq \pi_{2\ell+2}^+ \oplus \pi_{2\ell+2}^-
\quad
\text{for $\ell \in {\mathbb{N}}$.}
\]

Let $\pi$ denote the unitary principal series representation 
on the Hilbert space $L^2(G/P, {\mathcal{L}})$ 
where ${\mathcal{L}}={\mathcal{L}}_1^-$.  
The pull-back $\iota_{\lambda}^{\ast}$
 in Lemma \ref{lem:23062717} with $(\lambda, \varepsilon)=(1,-)$, 
\[
  f(\zeta_1, \zeta_2) \mapsto
  F(x, y, z)=\frac{2}{z-y}
  f(-\frac{y+z}{x+1}, -\frac{x+1}{z-y})
 = (\zeta_1-\zeta_2)f(\zeta_1,\zeta_2)
\]
 induces a unitary equivalence up to scaling:
\begin{equation}
\label{eqn:unitarytensor}
 \iota_1^{\ast} \colon 
 L^2(G/P, {\mathcal{L}})\widehat \otimes
 L^2(G/P, {\mathcal{L}})
  \overset \sim \rightarrow L^2(G/H)
\end{equation}
 because $G/H$ is conull in $G/P \times G/P$.

Thus the fusion rule of the left-hand side 
({\it{cf}}.\ Repka \cite{Re79})
 is equivalent to the Plancherel formula
 of $\operatorname{d S}^2$
 given in \eqref{eqn:Pl}.

The following theorem is a key 
 to the proof of Theorem \ref{thm:23062904}
 for an explicit embedding of discrete series representations.  
We note 
 that an analogous extension statement 
is not true 
 if we drop the square integrability assumption 
 of eigenfunctions in Theorem \ref{thm:230626_SB}.

\begin{theorem}
\label{thm:230626_SB}
Any $K$-finite function 
 of the discrete series for $G/H$
 extends to a real analytic section for ${\mathcal{L}} \boxtimes {\mathcal{L}}$
 over $G/P \times G/P$
 via \eqref{eqn:unitarytensor}.  
\end{theorem}

\begin{proof}
By \eqref{eqn:unitarytensor}, 
 the Plancherel formula for $G/H$ may be 
interpreted
 as the fusion rule of $\pi \widehat \otimes \pi$.  
Let ${\mathbb{H}}(\Pi_+)$ and ${\mathbb{H}}(\Pi_-)$
 denote the Hardy space
 for the upper half plane $\Pi_+$
 and the lower one $\Pi_-$, 
 respectively.  
Then one has a unitary equivalence
$\pi \simeq {\mathbb{H}}(\Pi_+) \oplus {\mathbb{H}}(\Pi_-)$, 
 and the discrete part and the continuous part in \eqref{eqn:Pl}
 are explained as 
\begin{align}
\label{eqn:HH}
   {\mathbb{H}}(\Pi_{\varepsilon}) \,\widehat\otimes \, {\mathbb{H}}(\Pi_{\varepsilon})
   \,\,\simeq\,\,& {\sum_{\ell=0}^{\infty}}
\raisebox{2ex}{\text{\scriptsize{$\oplus$}}}
   \pi_{2\ell+2}^{\varepsilon}
\qquad
 \text{$\varepsilon=+$ or $-$}, 
\\
\notag
   {\mathbb{H}}(\Pi_+)\, \widehat\otimes \,{\mathbb{H}}(\Pi_-)
   \,\,\simeq\,\,& \int_{(0,\infty)}^{\oplus} \varpi_{1+\sqrt{-1}\nu} d \nu.  
\end{align}
The tensor product 
 ${\mathbb{H}}(\Pi_+)\,\widehat\otimes\, {\mathbb{H}}(\Pi_-)$ is unitarily isomorphic 
 to $L^2(G/K)$, 
 and does not contain discrete spectrum 
 in the fusion rule.  
On the other hand, 
 any discrete series for $\operatorname{d S}^2$
 arises from the $K$-admissible tensor product 
 ${\mathbb{H}}(\Pi_\varepsilon) \,\widehat \otimes\, {\mathbb{H}}(\Pi_\varepsilon)$
 (\cite{xkAnn98}), 
 hence any $K$-finite vector $f$ becomes $(K \times K)$-finite
 by \cite{xkInvent98}.  
Since the direct product group $K \times K$ acts transitively 
 on $G/P \times G/P$, 
 the function $f \in L^2(G/H)$ extends to a real analytic section $\widetilde f$
 over $G/P \times G/P$ via \eqref{eqn:unitarytensor}.  
\end{proof}

\begin{example}
\label{ex:52}
Let $f_{\ell}$ be a function on $\operatorname{dS}^2 \simeq G/H$ given by
\[
    f_{\ell}^{\pm}(x,y,z)
  :=\left(\frac{\sqrt{-1}}{x\pm\sqrt{-1} y}\right)^{\ell+1}.  
\]
Then it belongs to a $K$-finite function 
 in $L^2(G/H, {\mathcal{M}}_{2 \ell+2})$, 
 giving a minimal $K$-type in $\pi_{2\ell+2}^{\pm}$, 
 and extends to an analytic section
\[
   \widetilde{f_{\ell}^{\pm}}(\zeta_1, \zeta_2) 
  =(\zeta_1-\zeta_2)^{\ell}(\zeta_1 \pm \sqrt{-1})^{-\ell-1}
   (\zeta_2\pm\sqrt{-1})^{-\ell-1}, 
\]
 for the line bundle ${\mathcal{L}} \boxtimes {\mathcal{L}}$
 over $G/P \times G/P$, 
 by \eqref{eqn:23062715} and \eqref{eqn:23062714}.

As shown in \cite[Prop.\ 2.28]{KPinv}, 
 $\widetilde{f_{\ell}^{+}}$ gives a minimal $K$-type of $\pi_{2\ell+2}^+$
 in the decomposition \eqref{eqn:HH}.  
Likewise for $\widetilde{f_{\ell}^-}$ in $\pi_{2\ell+2}^-$.  

\end{example}

\subsection{Embedding of discrete series for $\operatorname{d S}^2$}
\label{subsec:RCdisc}
~~~
\newline
We recall from \eqref{eqn:GPGP}
 that $\operatorname{d S}^2$ is realized
 as an open dense subset of $G/P \times G/P$, 
 with the boundary being isomorphic to $\operatorname{diag}(G/P)$.  
\begin{theorem}
[embedding of discrete series]
\label{thm:23062904}
The Rankin--Cohen brackets $R_{\ell}$ induces
 an injective $({\mathfrak{g}}, K)$-homomorphism from discrete series representations
 $\pi_{2\ell+2}^+$ and $\pi_{2\ell+2}^-$
 for the de Sitter space $\operatorname{d S}^2$
 into the principal series representation $C^{\infty}(G/P, {\mathcal{L}}_{2\ell+2})$
 for every $\ell \in {\mathbb{N}}$.  
\end{theorem}

\begin{proof}
Any $K$-finite function $f$ in $L^2(G/H, {\mathcal{M}}_{2\ell+2})$ extends
 to a real analytic section $\widetilde f$
 for the line bundle ${\mathcal{L}} \boxtimes {\mathcal{L}} \to G/P \times G/P$
 by Theorem \ref{thm:230626_SB}.  
Therefore $f \mapsto R_{\ell} \widetilde f$ is a well-defined 
 $({\mathfrak{g}}, K)$-homomorphism from 
$
  L^2(G/H, {\mathcal{M}}_{2\ell+2})_K 
$ 
 to $C^{\infty}(G/P, {\mathcal{L}}_{2 \ell+2})_K$.

Finally, 
 let us prove
 that this map is injective.  
Since $
  L^2(G/H, {\mathcal{M}}_{2\ell+2})
$
 splits into irreducible representations 
 $\pi_{2\ell+2}^+$ and $\pi_{2\ell+2}^-$, 
 it suffices to show
\begin{equation}
\label{eqn:23062618}
R_{\ell} \widetilde{f_{\ell}^+} \ne 0, 
\quad
R_{\ell} \widetilde{f_{\ell}^-} \ne 0, 
\end{equation}
 where $f_{\ell}^{\pm}$ are defined 
 in Example \ref{ex:52}.  
Then the assertion \eqref{eqn:23062618} holds
 because 
\[
   (R_{\ell} \widetilde{f_{\ell}^+})(\zeta)=\frac{(2 \ell)!}{\ell !}(\zeta+\sqrt{-1})^{-2\ell-2} \ne 0, 
\]
 see \cite[Ex.\ 3.9]{KPgen}, 
 and likewise for $R_{\ell}\widetilde{f_{\ell}^-}$.  
\end{proof}

\section{Appendix: Hyperfunctions and the Riemann--Liouville integral}
\label{sec:Append}

Our key idea from \lq\lq{discrete}\rq\rq\ to \lq\lq{continuous}\rq\rq\
 in Section \ref{sec:SL}
 is to use the fractional power of normal derivative
 \eqref{eqn:tTRl}.  
In order to implement the classical idea 
 of the Riemann--Liouville integral 
 into the \lq\lq{generating operators}\rq\rq, 
 we utilize the theory of hyperfunctions.  
\begin{lemma}
\label{lem:23062118}
The following formul{\ae} hold as a meromorphic continuation
 of $\lambda \in {\mathbb{C}}$
 and an analytic continuation of $w \in {\mathbb{C}}$:
\begin{alignat*}{2}
\langle t_+^{\lambda}, \frac{1}{t+w}\rangle
=&
-\frac{\pi w^{\lambda}}{\sin \pi \lambda}
\quad
&&\text{if $w \not\in (-\infty, 0]$, }
\\
\langle t_-^{\lambda}, \frac{1}{t+w}\rangle
=&
\frac{\pi (-w)^{\lambda}}{\sin \pi \lambda}
\quad
&&\text{if $w \not\in [0, \infty)$. }
\end{alignat*}
\end{lemma}

\begin{proof}
Suppose $-1 < \operatorname{Re}\lambda < 0$.  
Then the following integral converges to the Beta function:
\[
  \int_0^{\infty} \frac{t^{\lambda}}{t+1} d t=
  B(\lambda+1, -\lambda)
  =\frac{-\pi}{\sin \pi \lambda}.  
\]
Suppose $w \in {\mathbb{C}}$
 with $\operatorname{Re}w >0$.  
Then the change of variables yields 
\[
\int_0^{\infty} \frac{t^{\lambda}}{t+w} d t=
 \int_\gamma \frac{(sw)^{\lambda}}{s+1}ds,
\]
 where the path $\gamma$ is given by $\{\frac t w: 0 \le t < \infty\}$.  
By the Cauchy integral formula, 
 one sees readily
 that the integral does not change
 if we replace the path $\gamma$ with $[0,\infty)$.  
Hence the first equality holds initially 
defined as the convergent integral 
 for $-1 <\operatorname{Re}\lambda<0$
 and $\operatorname{Re}w>0$, 
 and extends meromorphically 
 in $w \in {\mathbb{C}} \setminus (-\infty, 0]$
 and $\lambda \in {\mathbb{C}}$.  

The proof of the second statement is similar.  
\end{proof}

The sheaf ${\mathcal{B}}$ of hyperfunctions is defined
 as local cohomology group.  
In one dimensional case, 
 for an open set $U$ in ${\mathbb{R}}$, 
 ${\mathcal{B}}(U) \simeq {\mathcal{O}}(\widetilde U \setminus U)/{\mathcal{O}}(\widetilde U)$
 where $\widetilde U$ is any open set in ${\mathbb{C}}$
 containing $U$, 
 and this definition does not depend on the choice of $\widetilde U$
 \cite{xsato}.

Then $w^{\lambda} \in {\mathcal{O}}({\mathbb{C}} \setminus (-\infty,0])$
defines a hyperfunction
\[
(e^{\sqrt{-1}\pi \lambda}-e^{-\sqrt{-1}\pi \lambda})w_-^{\lambda}
=2\sqrt{-1} \sin \pi \lambda w_-^{\lambda}
\]
 as a \lq\lq{boundary value}\rq\rq\ \cite{xsato}, 
 and 
$(-w)^{\lambda}= (e^{-\sqrt{-1}\pi} w)^{\lambda}\in {\mathcal{O}}({\mathbb{C}} \setminus [0, \infty))$ defines
\[
   (e^{-\sqrt{-1}\pi \lambda}-e^{\sqrt{-1}\pi \lambda})w_+^{\lambda}
=-2\sqrt{-1} \sin \pi \lambda w_+^{\lambda}.  
\]

Hence Lemma \ref{lem:23062118} may be reinterpreted as below.  
\begin{lemma}
\label{lem:23062411}
As hyperfunctions that depend meromorphically on $\lambda \in {\mathbb{C}}$, 
 one has the following equations.  
\begin{equation*}
\langle t_{\pm}^{\lambda}, \frac{1}{t+w}\rangle
=
-2 \pi \sqrt{-1} w_{\mp}^{\lambda}.  
\end{equation*}
\end{lemma}

\vskip 1pc
\par\noindent
{\bf{Acknowledgement.}}
\newline
The article is an outgrowth of the invited address
 at the Nordic Congress of Mathematicians
 held on July, 2023 in Denmark.  
The author expresses his gratitude
 to J.\ Frahm, M.\ Pevzner, and G.\ Zhang for their warm hospitality.  
He also thanks Professor Kubo 
 for reading carefully the first draft, 
 and anonymous referees for their comments.  
This work was partially supported by the JSPS
 under the Grant-in Aid for Scientific Research (A) 
 (JP18H03669, JP23H00084).  

\vskip 1pc
\par\noindent
{\bf{Data availability. }}
\newline
Data sharing not applicable to this work 
 as no datasets were generated or analysed 
 during the current study.  

\vskip 1pc
\par\noindent
{\bf{Statement on conflict if interest.}}
\newline
On behalf of all authors, 
 the corresponding author states
 that there is no conflict of interest.

\vskip 3pc
\leftline
{Toshiyuki KOBAYASHI}
\leftline
{Graduate School of Mathematical Sciences, }
\leftline
{The University of Tokyo,
 3-8-1 Komaba, Meguro, 
Tokyo, 153-8914, Japan.}
\leftline
{toshi@ms.u-tokyo.ac.jp}
\end{document}